\documentclass[reqno]{amsart}

\usepackage[backend=biber, style=abbrv, sorting=none, bibstyle=numeric, citestyle=numeric, sorting=nyt,maxbibnames=99, doi=false, url=false]{biblatex}

\renewbibmacro{in:}{}
\bibliography{references.bib}
\usepackage{amssymb}
\usepackage{calrsfs}
\usepackage{graphics,graphicx, mathrsfs}
\usepackage{enumerate}
\usepackage{enumitem}
\usepackage{url}
\usepackage{xcolor}
\definecolor{vio}{rgb}{0.54, 0.17, 0.89}
\usepackage[ruled, lined, linesnumbered, longend]{algorithm2e}
\usepackage{hyperref}
\usepackage{titlesec}
\usepackage{siunitx}
\usepackage{caption}
\usepackage{longtable}
\usepackage[referable]{threeparttablex}
\allowdisplaybreaks

\newtheorem{theorem}{Theorem}[section]
\newtheorem{lemma}[theorem]{Lemma}

\numberwithin{equation}{section}

\theoremstyle{remark}
\newtheorem{remark}[theorem]{Remark}

\titleformat{\section}
  {\normalfont\large\bfseries\centering}{\thesection}{1em}{}

\titleformat{\subsection}
  {\normalfont\bfseries}{\thesubsection}{1em}{}  
\titleformat{\subsubsection}
  {\normalfont\bfseries}{\thesubsubsection}{1em}{}

\newcommand{\ceiling}[1]{\left\lceil#1\right\rceil}

\def\reals{\hbox{\rm I\kern-.18em R}}
\def\complexes{\hbox{\rm C\kern-.43em
\vrule depth 0ex height 1.4ex width .05em\kern.41em}}
\def\field{\hbox{\rm I\kern-.18em F}} %symbol for field

\newenvironment{section*}[2][A]{
  \section*{#2}
  \renewcommand\thesection{#1}
  \setcounter{theorem}{0}}{}

\begin{document}

\title[Divisor Problem]{On the Summatory Function of $d_3(n)$}

\author{Sebastian Tudzi}
\address{School of Science, UNSW Canberra, Australia}
\email{s.tudzi@unsw.edu.au}
\date\today
\subjclass[2020]{Primary: 11N99}
\keywords{ Divisor function, moments of the zeta-function, explicit results.}
\begin{abstract}
   In this article, we refine the method of our earlier work with N. Paloj{\"a}rvi to obtain a sharper explicit bound for the error term $\Delta_{3}(x)$ associated with the summatory function of $d_{3}(n)$. We prove that
    \begin{equation*}
 |\Delta_3(x)| <
\begin{cases}
0.6901\,x^{\frac{1}{2}}\log^{\frac{7}{2}}x, & 3.682\cdot 10^{31}\le x < 4.133\cdot 10^{87},\\[4pt]
0.2067\,x^{\frac{1}{2}}\log^{\frac{7}{2}}x, & 4.133\cdot 10^{87} \le x < 1.597\cdot 10^{98},\\[4pt]
0.1947\,x^{\frac{1}{2}}\log^{\frac{7}{2}}x, & x \ge 1.597\cdot 10^{98}.
\end{cases}
\end{equation*}
These explicit results improve the exponent of $x$ from $2/3$, due to Tudzi, and $859/1400$, due to Paloj{\"a}rvi  and Tudzi, to $1/2$, giving the best known bound for all $x\ge 3.682\cdot 10^{31}$.
\end{abstract}

\maketitle
%\blfootnote{\textit{Affiliation}: School of Science, The University of New South Wales Canberra, Australia}
%\blfootnote{\textit{Author}:Neea Paloj{\"a}rvi (n.palojarvi@unsw.edu.au) and Sebastian Tudzi (s.tudzi@unsw.edu.au)}
%\blfootnote{\textit{Key phrases}: Divisor function, analytic approach, explicit results.}
%\blfootnote{\textit{MSC classes}: 11N99}

\section{Introduction}
A classical problem in analytic number theory concerns the asymptotic behaviour of the summatory function
\[
T_3(x) := \sum_{n \leq x} d_3(n),
\]
where $d_3(n)$ represents the number of ways to write $n$ as a product of three positive integers. It is well known that
\begin{equation}\label{eq:asymptotic}
T_3(x) = x\!\left(\frac{(\log x)^2}{2} + (3\gamma - 1)\log x + 3\gamma^2 - 3\gamma - 3\gamma_1 + 1\right) + \Delta_3(x),
\end{equation}
where the main term on the right-hand side of \eqref{eq:asymptotic} is denoted by $xP_{3}(\log x)$, $\gamma$ and $\gamma_1$ are the Euler--Mascheroni constant and the first Stieltjes constant respectively, and $\Delta_3(x)$ is the corresponding error term (see \cite{MR2257286} and \cite{MR5085189}). It is conjectured that $\Delta_3(x)=O(x^{1/3+\varepsilon})$ for every $\varepsilon>0$ (see \cite[p.~320]{MR882550}). Motivated in part by this conjecture, considerable attention has been devoted to obtaining explicit upper bounds for $\Delta_{3}(x)$.

Several authors have established explicit results for the error term in \eqref{eq:asymptotic} using both elementary methods, such as the Dirichlet convolution method, and analytic techniques. Using the Dirichlet convolution method, Bordell{\'e}s \cite[Lemma~3.2]{MR2257286} and Tudzi \cite[Theorem~1.1]{MR5085189} proved respectively that
\begin{equation}\label{Bord}
    |\Delta_{3}(x)| \leq 2.36\, x^{2/3} \log x,\quad x>670
\end{equation}
and
\begin{equation}\label{eq:Tudzi}
|\Delta_3(x)| < 2.968\, x^{2/3} \log^{1/3} x, \quad x\ge 2.
\end{equation}
The current best known explicit result in \eqref{eq:Tudzi} holds on a wider range of $x$ than \eqref{Bord}, and its logarithmic saving makes it sharper than \eqref{Bord} throughout the latter's entire range of validity, despite its larger constant.

In a recent article, Paloj{\"a}rvi and Tudzi \cite{palojarvi2026explicit} employed analytic methods to obtain  explicit estimates for $\Delta_{k}(x)$, where $k\ge 2$.  For the  case $k=3$, they combined second moment estimates for the Riemann zeta function on the critical line with bounds for $\zeta(1/2 + it)$ to derive explicit upper bounds for $\Delta_{3}(x)$ in different ranges of $x$. More precisely, the main results for $\Delta_{3}(x)$ in \cite[Table 7]{palojarvi2026explicit} are as follows:
\begin{equation}\label{nana11}
|\Delta_3(x)| <
\begin{cases}
3.6127\,x^{\frac{93}{140}}\log^{3}x, & 176749\le x < 4.1332\cdot 10^{87},\\[4pt]
0.2797\,x^{\frac{13}{21}}\log^{3}x, & 4.1332\cdot 10^{87} \le x < 1.6\cdot 10^{98},\\[4pt]
0.2475\,x^{\frac{859}{1400}}\log^{3}x, & x \ge 1.6\cdot 10^{98}.
\end{cases}
\end{equation}
 From this, we observe that for $2\le x<1.601\cdot 10^{
98}$, the bound in \eqref{eq:Tudzi} remains the strongest known estimate, while for $x\ge1.601\cdot 10^{98}$,
 the final bound in \eqref{nana11} provides a sharper result, improving upon all previously known estimates in this range.

In the present paper, we adopted the approach similar to that in \cite{palojarvi2026explicit}, with suitable modifications and improvements. We improve upon the existing explicit estimates by employing the third moment estimate for the Riemann zeta function on the critical line, combined with convexity bounds for $\zeta(s)$.  We also provide new estimates for $|\zeta(\sigma+it)|$ with $\sigma\in [1/2,c]$ for some $c>1$. 

The paper is organised as follows. In Section \ref{sec:PreBound}, we provide estimates for the relevant divisor sums by treating separately the cases where $x$ is a half-integer and $x$ is an integer. Also, we include estimate for $\zeta(s)$ in the critical strip. In Section \ref{sec:momentEstimates}, we establish integral estimates using pointwise bounds and third moment estimates for $\zeta(1/2+it)$. Section \ref{sec:Effective} presents the effective version of the  result, while Section \ref{Proof} contains the proof of the main theorem. Finally, Section \ref{Conc} provides concluding remarks.

The following theorem provides the sharpest bound that can be obtained using our method. A variant with weaker constants extends the validity of the bound to a wider range of $x$, which we omit here.
\begin{theorem}\label{Final}
    We have 
   \begin{equation*}
   |\Delta_3(x)| <
\begin{cases}
0.6901\,x^{\frac{1}{2}}\log^{\frac{7}{2}}x, & 3.682\cdot 10^{31}\le x < 4.133\cdot 10^{87},\\[4pt]
0.2067\,x^{\frac{1}{2}}\log^{\frac{7}{2}}x, & 4.133\cdot 10^{87} \le x < 1.597\cdot 10^{98},\\[4pt]
0.1947\,x^{\frac{1}{2}}\log^{\frac{7}{2}}x, & x \ge 1.597\cdot 10^{98}.
\end{cases}
\end{equation*}
\end{theorem}
These bounds go beyond a numerical improvement in the exponent of $x$ over the bounds in \eqref{eq:Tudzi} and \eqref{nana11}, at a modest cost of the power of $\log x$. These improvements hold for all $x\ge 3.682\cdot 10^{31}$ and constitute the strongest explicit bounds currently known in this range.

\section{Preliminary estimates}
This section gathers the preliminary estimates needed in the proofs of our main results. We also include bounds for the Riemann zeta-function in the critical strip, which will play a central role in the estimation of the integrals arising in our arguments. 
\label{sec:PreBound}
\subsection{Integer and half-Integer cases}
Here, we establish an estimate for studying the behaviour of $T_3(x)$, with the argument split according to whether $x$ is an integer or half-integer.
\begin{lemma}
\label{lemma:TkIntegral}(\cite[Lemma 2.1]{palojarvi2026explicit}, $k=3$)
Assume $T>c>0$. If $x$ is a half-integer, then we have
\begin{equation}
\label{eq:TkSum}
    \left|T_3(x)-\frac{1}{2\pi i}\int_{c-iT}^{c+iT} \zeta^3(w)\frac{x^w}{w} \, \mathrm{d}w\right| <\frac{1+\pi}{\pi}\cdot\frac{x^c}{T}\sum_{n\ge1} \frac{d_3(n)}{n^c \left|\log{\frac{x}{n}}\right|}. 
\end{equation}
Moreover, if $x$ is an integer, then we have
\begin{multline}
\label{eq:integerdSumFirst}
    \left|T_3(x)-\frac{1}{2\pi i}\int_{c-iT}^{c+iT} \zeta^3(w)\frac{x^w}{w} \, \mathrm{d}w\right| \\
    <\frac{1+\pi}{\pi}\cdot\frac{x^c}{T}\sum_{\substack{n\ge1 \\ n \neq x}} \frac{d_3(n)}{n^c \left|\log{\frac{x}{n}}\right|}+\frac{d_3(x)}{2\pi}\left(\pi+\frac{2Tc}{T^2-c^2}\right). 
\end{multline}
\end{lemma}
This follows from \cite[Lemma 2.1]{palojarvi2026explicit} upon setting $k=3$.

To estimate the sum on the right-hand side of \eqref{eq:TkSum}, we first define the functions $u_{i},\, i=1,\hdots,6,$ which will be used throughout the subsequent lemma.
We define
\begin{align}
&u_{1}(x,\varepsilon,c,a):= \frac{ (ax)^{\frac{1.5914\log 3}{\log\log (ax)}}}{(ax)^{c}}+\frac{c(ax)^{1-c}}{c-1}\Bigl(\frac{\log^{2}(ax)}{2}+\frac{(\alpha(c-1)+1)\log(ax)}{(c-1)}+\notag\\&\frac{(\beta(c-1)^{2}+\alpha(c-1)+1)}{(c-1)^{2}}\Bigr)+\frac{2.968(\log ax)^{\frac{1}{3}}}{(ax)^{c-2/3}(3c-2)}\left(\frac{9c}{(3c-2)(\log ax)^{\frac{5}{3}}}+3c\right),\label{u3}\\
    &u_{2}(x,\varepsilon,c,a):= \frac{1}{\log a}\Bigl(
  \frac{2.968\log^{\frac{1}{3}}2}{2^{c-2/3}(3c-2)}\left(\frac{3c}{(3c-2)\log^{\frac{5}{3}}2}+6c-2\right)+u_{1}(x,\varepsilon,c,a)+\notag\\&\quad\hspace{1.5cm}\frac{(c+1)\left(\frac{\log^{2}2}{2}+\alpha\log 2+\beta\right)}{2^{c}(c-1)}+\frac{2c((c-1)(\log 2+\alpha)+1)}{2^{c}(c-1)^{3}}+1
    \Bigr),\label{m1rr}\\
    &u_{3}(x,a):=\log\left(x-\frac{x}{a}-\frac{1}{2}\right)+\gamma+\frac{1}{(2-2/a)x-1},\label{m2rr}\\
     &u_{4}(x,a):=\log(x(a-1))+\gamma+\frac{1}{2(a-1)x},\label{m3rr}\\
     &u_{5}(x,a):=\log\left(x-\frac{x}{a}\right)+\gamma+\frac{1}{2(1-1/a)x}-\frac{1}{2}+\frac{1}{2a},\label{m4rr}\\
     &u_{6}(x,a):=\log(x(a-1))+\gamma+\frac{1}{2(a-1)x}+\frac{1}{2}.\label{m5rr}\\\notag
\end{align}
\begin{lemma}
\label{lemma:dkGeneral}
Let $\varepsilon\in(0,17.590)$, 
$c$ be a real number such that $c \in (1+\varepsilon, 18.590)$  and $a > 1$. Assume
$x > 3a$, $x-1/2\ge e^{e}$, and $x(1-1/a)\geq 3/2$.

If $x$ is a half-integer, then
\begin{align}\label{halfintbound}
    \sum_{n\ge1} \frac{d_3(n)}{n^c \left|\log{\frac{x}{n}}\right|}
    & <u_{2}(x,\varepsilon, c, a)
    +\left(2x-\frac{1}{2}\right)\left(\frac{x}{a}\right)^{\varepsilon-c}u_{3}(x,a)+\notag\\ &\quad(a+1)x\left(x+\frac{1}{2}\right)^{\varepsilon-c}u_{4}(x,a)+(4x-1)(x-1/2)^{\epsilon-c}.
\end{align}
If $x$ is an integer, then
\begin{align}\label{out}
    \sum_{\substack{n\ge1 \\ n \neq x}} \frac{d_3(n)}{n^c \left|\log{\frac{x}{n}}\right|}
    &< u_{2}(x,\varepsilon, c, a)
    + \left(x-\frac{1}{2}\right)\left(\frac{x}{a}\right)^{\varepsilon-c}u_{5}(x,a)+\notag\\ &\quad
    (x+1)^{\varepsilon-c}\left(x+\frac{1}{2}\right)u_{6}(x,a).
\end{align}
\end{lemma}

\begin{proof}
If $n < x/a$ or $n \geq ax$, then $|\log(x/n)| \ge \log a$, so
\begin{equation}
    \sum_{n < x/a} \frac{d_{3}(n)}{n^c\left|\log\frac{x}{n}\right|}
    + \sum_{n \geq ax} \frac{d_{3}(n)}{n^c\left|\log\frac{x}{n}\right|}
    \leq \frac{1}{\log a}\left(\sum_{n < x/a} \frac{d_{3}(n)}{n^c}
    + \sum_{n \geq ax} \frac{d_{3}(n)}{n^c}\right).
    \label{eq:nLargeSmall}
\end{equation}
From \cite[Theorem~1.1]{MR5085189}, we have
\begin{align}\label{convo3}
    T_{3}(x)&\leq x\!\left(\frac{\log^{2}x}{2} + \alpha\log x + \beta\right)
    + 2.968x^{\frac{2}{3}}\log^{\frac{1}{3}}x\nonumber\\
    &=: xP_3(\log x) + 2.968\,x^{\frac{2}{3}}\log^{\frac{1}{3}}x
\end{align}
for all $x \geq 2$, where $\alpha = 3\gamma - 1$ and
$\beta = 3\gamma^{2} - 3\gamma_{1} - 3\gamma + 1$.
Applying partial summation to the first sum on the right-hand side of \eqref{eq:nLargeSmall} gives
\begin{align}\label{Asss}
    \sum_{n < x/a} \frac{d_{3}(n)}{n^c}
    &= 1+\sum_{2\le n < x/a} \frac{d_{3}(n)}{n^c}\notag\\
    &\le 1+ \frac{T_{3}(x/a)}{(x/a)^c}-\frac{T_{3}(2)}{2^{c}}+c\int_{2}^{x/a}\frac{T_{3}(t)}{t^{c+1}}\,\mathrm{d}t.
\end{align}
 We define the second and the fourth terms in \eqref{Asss} as $A_{1}$ and $A_{2}$ respectively. 

Now, using \eqref{convo3}, we have 
\begin{equation}\label{dgbd}
   A_{1}\le  \frac{\,P_{3}(\log(\frac{x}{a}))}{(x/a)^{c-1}}+\frac{2.968\log^{1/3}(x/a)}{(x/a)^{c-2/3}}
\end{equation}
and 
\begin{align}\label{goo}
    A_2\le c\left(\int_{2}^{x/a}\frac{\frac{\log^{2}t}{2}+\alpha\log t+\beta}{t^c}\, \mathrm{d}t+2.968\int_{2}^{x/a}\frac{\log^{1/3}t}{t^{c+1/3}}\,\mathrm{d}t\right).
\end{align}
The first term in \eqref{goo} gives 
\begin{align}\label{googong}
 &=-\frac{(\frac{x}{a})^{1-c}c}{c-1}\left(\frac{\log^{2}(\frac{x}{a})}{2}+\frac{(\alpha(c-1)+1)\log(\frac{x}{a})}{(c-1)}+\frac{(\beta(c-1)^{2}+\alpha(c-1)+1)}{(c-1)^{2}}\right)+\notag\\&\,\quad\,\frac{2^{1-c}c}{c-1}\left(\frac{\log^{2} 2}{2}+\frac{(\alpha(c-1)+1)\log2}{(c-1)}+\frac{(\beta(c-1)^{2}+\alpha(c-1)+1)}{(c-1)^{2}}\right).   
\end{align}
Since $\alpha$ and $\beta$ are positive, $c>1$ and $x/a>3$, it follows that the sum of the first terms in \eqref{dgbd} and \eqref{googong}, together with the main term of the third term in \eqref{Asss}, is bounded above by
\begin{align*}
\le&\frac{(c+1)\left(\frac{\log^{2}2}{2}+\alpha\log 2+\beta\right)}{2^{c}(c-1)}+\frac{2c((c-1)(\log 2+\alpha)+1)}{2^{c}(c-1)^{3}}.
\end{align*}
Applying integration by parts to the second term in \eqref{goo} gives
\begin{equation}\label{kaka}
    = \frac{2.968c}{c-2/3}\left(\frac{\log^{1/3}2}{2^{c-2/3}}-\frac{\log^{1/3}(\frac{x}{a})}{(\frac{x}{a})^{c-2/3}}\right)+\frac{2.968c}{3c-2}\int_{2}^{x/a}\frac{1}{t^{c+1/3}\log^{2/3}t}\,\mathrm{d}t.
\end{equation}
Notice that, since $c>1$, the sum of the second terms in \eqref{dgbd} and \eqref{kaka} is non-positive and $\log^{2/3}(x/a)>\log^{2/3}3$ for all $x>3a$. Hence, we bound the sum of the second terms in \eqref{dgbd} and \eqref{goo}, together with the error term of the third term in \eqref{Asss} as follows:
\begin{align*}
\le \frac{2.968\log^{\frac{1}{3}}2}{2^{c-2/3}(3c-2)}\left(\frac{3c}{(3c-2)\log^{\frac{5}{3}}2}+6c-2\right).
\end{align*}

Therefore, we have 
\begin{align}\label{ghhd}
      \sum_{ n < x/a} \frac{d_{3}(n)}{n^c}&\leq\frac{2^{1-c}c}{c-1}\left(\frac{\log^{2}2}{2}+\frac{(\alpha(c-1)+1)\log2}{(c-1)}+\frac{(\beta(c-1)^{2}+\alpha(c-1)+1)}{(c-1)^{2}}\right)\notag\\ &\quad+ \frac{2.968\log^{\frac{1}{3}}2}{2^{c-2/3}(3c-2)}\left(\frac{3c}{(3c-2)\log^{\frac{5}{3}}2}+6c-2\right)+1.
\end{align}

For the second sum on the right-hand side of \eqref{eq:nLargeSmall}, we write 
\begin{equation}\label{eoco}
    \sum_{n \geq ax} \frac{d_{3}(n)}{n^c}=\frac{d_{3}(ax)}{(ax)^c}+\sum_{n >ax} \frac{d_{3}(n)}{n^c}
\end{equation}
We apply the bound 
\begin{equation}\label{fge}
    d_{3}(n)\le n^{\frac{1.5914\log 3}{\log\log n}}
\end{equation}
from \cite[Section 1]{teo2025} to estimate the first term on the right-hand side of \eqref{eoco} and use partial summation together with \eqref{convo3} to obtain a bound for the last term. This gives
\begin{align}\label{fghe}
    \sum_{n \geq ax}& \frac{d_{3}(n)}{n^c}\leq\frac{ (ax)^{\frac{1.5914\log 3}{\log\log (ax)}}}{(ax)^{c}}+\frac{c(ax)^{1-c}}{c-1}\Bigl(\frac{\log^{2}(ax)}{2}+\frac{(\alpha(c-1)+1)\log(ax)}{(c-1)}+\notag\\&\frac{(\beta(c-1)^{2}+\alpha(c-1)+1)}{(c-1)^{2}}\Bigr)+ \frac{2.968\log^{\frac{1}{3}}(ax)}{(ax)^{c-2/3}(3c-2)}\left(\frac{9c}{(3c-2)\log^{\frac{5}{3}}(ax)}+3c\right)
\end{align}
for $c \in (1+\varepsilon, 18.589)$. Note that one could obtain a sharper bound for \eqref{fghe} by directly applying partial summation and the bound in \eqref{convo3}. However, the resulting improvement is negligible, so the pointwise bound suffices. Now, substituting \eqref{ghhd} and \eqref{fghe} into \eqref{eq:nLargeSmall} yields
\begin{equation*}
    \sum_{n < x/a} \frac{d_3(n)}{n^c\left|\log\frac{x}{n}\right|}
    + \sum_{n \geq ax} \frac{d_3(n)}{n^c\left|\log\frac{x}{n}\right|}
    \leq u_{2}(x,\varepsilon,c,a).
\end{equation*}

Next, we consider the range $x/a \leq n \leq ax$. For the case where $x$ is a half-integer, we  write
\begin{align*}
    \sum_{x/a \leq n \leq ax} \frac{d_{3}(n)}{n^{c}\left|\log\frac{x}{n}\right|}
    &= \sum_{x/a \leq n \leq x-3/2}
    \frac{d_{3}(n)}{n^{c}\log(x/n)}
    + \sum_{x+1/2 \leq n \leq ax}
    \frac{d_{3}(n)}{n^{c}\log(n/x)}\\
    &:= S_{1} + S_{2}.
\end{align*}
To estimate $S_1$ and $S_{2}$, we adapt the same argument as in the corresponding part of the proof in \cite[Lemma 2.5]{palojarvi2026explicit} to the case $d_{3}(n)$ as in \eqref{fge}. Hence, we have 
\begin{equation*}
    S_1\le\left(2x-\frac{1}{2}\right)\left(\frac{x}{a}\right)^{\varepsilon-c} u_{3}(x,a)
\end{equation*}
and 
\begin{equation*}
    S_{2}\le  (a+1)x\left(x+\frac{1}{2}\right)^{\varepsilon-c}u_{4}(x,a),
\end{equation*}
where $u_{3}$ and $u_{4}$ are as defined in \eqref{m2rr} and \eqref{m3rr} respectively. 

Again, similar to the proof in \cite[Lemma 2.5]{palojarvi2026explicit}, if $n=x-1/2$, we have 
\begin{equation}\label{dkboundhalf}
    \frac{d_{3}(n)}{n^{c}\left|\log\frac{x}{n}\right|}\le (4x-1)(x-1/2)^{\epsilon-c}.
\end{equation}
Hence, the bound for half-integer $x$ now follows by combining the estimates for $u_1$, $u_2$, $u_3$, and \eqref{dkboundhalf}.

We now consider the case where $x$ is a positive integer. Since $n \neq x$, the middle range splits into $x/a \leq n \leq x-1$ and $x+1 \leq n \leq ax$, giving sums $S_3$ and $S_4$ respectively. For $S_3$ and $S_{4}$, we adapt the same argument as in the corresponding part of the proof in \cite[Lemma 2.5]{palojarvi2026explicit}. Hence, we have the following:
\begin{equation*}
   S_3\le\left(x-\frac{1}{2}\right)\left(\frac{x}{a}\right)^{\varepsilon-c}u_{5}(x,a)
\end{equation*}
and
\begin{equation*}
S_{4}\le (x+1)^{\varepsilon-c}\left(x+\frac{1}{2}\right)u_{6}(x,a),
\end{equation*}
where $u_{5}$ and $u_{6}$ are as defined in \eqref{m4rr} and \eqref{m5rr} respectively. The bound for integer $x$ follows by combining the estimates for
$u_1$, $u_4$, and $u_5$.
\end{proof}
\subsection{Auxiliary bounds for $\zeta(s)$ in the critical strip}
The representation obtained in Lemma 2.1 expresses $T_
3(x)$ in terms of an integral. Our next step is to identify the components contributing to this integral and establish suitable bounds for each of them.
 \begin{lemma}(\cite[Lemma 3.1]{palojarvi2026explicit}, $k=3$)
\label{lemma:IntegralsToBeEstimated}
Let $c>1$, $b<1$ and $T>0$. If $b>0$, then we have
\begin{multline}
\label{eq:zetakRectangle}
\frac{1}{2\pi i}\int_{c-iT}^{c+iT} \zeta^3(w)\frac{x^w}{w} \, \mathrm{d}w \\
=xP_3(\log{x})-\frac{1}{2\pi i}\left(\int_{c+iT}^{b+iT}+\int_{b+iT}^{b-iT}+\int_{b-iT}^{c-iT}\right) \zeta^3(w)\frac{x^w}{w} \, \mathrm{d}w.
\end{multline}
\end{lemma}

The following lemmas provide explicit convexity bound for $\zeta(\sigma+it)$ in  the regions $1/2\le \sigma\le 5/7$ and $5/7\le \sigma\le c$.
\begin{lemma}
\label{lemma:zeta_half}
Assume $s := \sigma + it$, $t \geq t_0 \geq 3$, and $\sigma \in [1/2, 5/7]$. Let $\lambda :=(10 - 14\sigma)/3$. Then
\begin{equation*}
    |\zeta(s)| \leq \nu_{1}(\sigma, t_0)\, t^{\frac{\lambda}{6} + \frac{1-\lambda}{14}} \log t,
\end{equation*}
where
\begin{equation*}
    \nu_{1}(\sigma, t_0) := 0.611^{\lambda} \cdot 1.546^{1-\lambda} \cdot \left(1+\frac{\log(1 + 1.477t_{0}^{-2})}{2\log t_{0}}\right)\cdot\left(1 + \frac{1.477}{t_0^2}\right)^{\frac{1}{2}\left(\frac{\lambda}{6} + \frac{1-\lambda}{14}\right)}.
\end{equation*}
\end{lemma}

\begin{proof}
We apply \cite[Lemma 3]{Leong2024} with $f(s):=(s-1)\zeta(s)$. We begin by establishing bounds for $f(s)$ on each boundary line. On the line $\Re(s)=1/2$, it follows from \cite[Theorem~1.1(b)]{R2026} that
\begin{equation*}
    |\zeta(1/2 + it)| \leq 0.611\, t^{1/6} \log t, \quad t \geq t_0 \geq 3.
\end{equation*}
Since $|s-1|\leq |Q_{0}+s|$, $t\leq |Q_{0}+s|$, and $\log t\leq \log|Q_{0}+s|$, where $Q_{0}+1/2>1$, we deduce that
\begin{equation*}
    |f(1/2 + it)| \leq 0.611\, |Q_0 + s|^{7/6}(\log|Q_0 + s|).
\end{equation*}
For $\Re(s) = 5/7$, by \cite[Theorem 1.1]{MR4693232} with $k = 4$ we have 
\begin{equation*}
    |\zeta(5/7 + it)| \leq 1.546\, t^{1/14} \log t, \quad t \geq t_0 \geq 3.
\end{equation*}
 Hence, we have
\begin{equation*}
    |f(5/7 + it)| \leq 1.546\, |Q_0 + s|^{15/14}(\log|Q_0 + s|).
\end{equation*}

Next, we verify the conditions of \cite[Lemma 3]{Leong2024} with $A = 0.611$, $B = 1.546$, 
$\alpha_1 = 7/6$, $\beta_1 = 15/14$, and $\alpha_2 = \beta_2 = 1$. 
Notice that the inequalities $\alpha_1 \geq \beta_1$ is satisfied. Hence,
with $a = 1/2$, $b = 5/7$, and $b-a = 3/14$, we obtain
\begin{align*}
    |f(s)| &\leq \left(0.611\,|Q_0+s|^{7/6}\log|Q_0+s|\right)^{\lambda} 
              \cdot \left(1.546\,|Q_0+s|^{15/14}\log|Q_0+s|\right)^{1-\lambda} \\
           &= 0.611^{\lambda} \cdot 1.546^{1-\lambda} 
              \cdot |Q_0+s|^{\frac{7\lambda}{6} + \frac{15(1-\lambda)}{14}} 
              \cdot \log|Q_0+s|.
\end{align*}
Recovering $\zeta(s) = f(s)/(s-1)$ and using $|s-1| \geq t$ for $t \geq t_0$, we obtain
\begin{equation*}
    |\zeta(s)| \leq 0.611^{\lambda} \cdot 1.546^{1-\lambda} 
    \cdot |Q_0+s|^{\frac{\lambda}{6} + \frac{1-\lambda}{14}} \cdot \log|Q_0+s|.
\end{equation*}

Finally, taking $Q_{0}=0.501$ we have
\begin{equation*}
    |Q_0+s| \leq \left(1 + \frac{(Q_{0}+5/7)^2}{t_0^2}\right)^{1/2}t\le \left(1 + \frac{1.477}{t_0^2}\right)^{1/2}t
\end{equation*}
and 
\begin{equation*}
    \log|Q_0+s| \leq  \left(1+\frac{\log(1 + 1.477t_{0}^{-2})}{2\log t_{0}}\right)\log t.
\end{equation*}
This completes the proof.
\end{proof}

\begin{lemma}
\label{lemma:zetaMiddle2}
Assume $c > 1$, $s := \sigma + it$, $t \geq t_0 \geq 3$ and $\sigma \in [5/7, c]$. 
Then
\begin{equation*}
    |\zeta(s)| \leq \nu_{2}(c, \sigma, t_0)\, t^{\frac{c-\sigma}{14(c-5/7)}} \log t,
\end{equation*}
where
\begin{multline*}
    \nu_{2}(c, \sigma, t_0) := 1.546^{\frac{c-\sigma}{c-5/7}} 
    \left(\frac{\zeta(c)}{\log t_0}\right)^{\frac{\sigma - 5/7}{c-5/7}}
    \left(1+\frac{\pi}{2\log{t_0}}+\frac{\pi(c+1.31)^2}{4t_0(\log{t_0})^2}+\frac{c+1.31}{2t_0^2\log{t_0}}\right) \\
    \cdot \left(\frac{c+1.31+t_0}{t_0}\right)^{\frac{c-\sigma}{14(c-5/7)}+1}. 
\end{multline*}
\end{lemma}

\begin{proof}
The proof follows analogously to that of \cite[Lemma~3.2]{palojarvi2026explicit} and \cite[Second example]{MR5016717}, except that the left boundary $\sigma=\tfrac12$ is replaced by $\sigma=\tfrac57$, and \cite[Theorem~1.1]{MR4693232} with $k=3$ is used in place of \cite[Theorem~1.1(a)]{R2026}.
\end{proof}

In the next lemma, we derive estimates for the first and last integrals on the right-hand side of \eqref{eq:zetakRectangle}.

\begin{lemma}\label{Lemu7}
Assume $T_1 \geq T \geq T_0 \geq 3$ and $c > 1$. Then
\begin{equation*}
    \left|\frac{1}{2\pi i}\left(\int_{c+iT}^{1/2+iT} + \int_{1/2-iT}^{c-iT}\right)\zeta^3(w)\frac{x^w}{w}\,\mathrm{d}w\right| \leq u_{7}\!\left(x,\tfrac{1}{2},c,T,T_0,T_1\right),
\end{equation*}
where
\begin{multline*}
    u_{7}\!\left(x,\tfrac{1}{2},c,T,T_0,T_1\right) := 
    \frac{(\log T_1)^3}{T\pi}
    \!\left[
        \frac{3}{14}
        \max\left\{
            \nu_{1}\!\left(\tfrac{1}{2},T_0\right)^3 T^{\frac{1}{2}}x^{\frac{1}{2}},\; 
            \nu_{1}\!\left(\tfrac{5}{7},T_0\right)^3 
            T^{\frac{3}{14}}x^{\frac{5}{7}}
        \right\}\right.\\
        \left.+\left(c-\frac{5}{7}\right)
        \max\left\{
            \nu_{2}\!\left(c,\tfrac{5}{7},T_0\right)^3 
            T^{\frac{3}{14}}x^{\frac{5}{7}},\; 
            \nu_{2}(c,c,T_0)^3 x^c
        \right\}
    \right].
\end{multline*}
\end{lemma}

\begin{proof}
We follow the proof in \cite[Lemma 3.6]{palojarvi2026explicit}. Now, we write
\begin{equation*}
    \left|\frac{1}{2\pi i}\left(\int_{1/2+iT}^{c+iT} + \int_{1/2-iT}^{c-iT}\right)\zeta^3(s)\frac{x^s}{s}\,\mathrm{d}s\right| 
    \leq \frac{1}{\pi}\int_{1/2}^{c}\frac{|\zeta(\sigma+iT)|^3 x^\sigma}{T}\,\mathrm{d}\sigma.
\end{equation*}
We split the 
integral on the right-hand side as follows:
\begin{equation}\label{priyam}
    \frac{1}{\pi}\int_{1/2}^{c}\frac{|\zeta(\sigma+iT)|^3 x^\sigma}{T}\,\mathrm{d}\sigma 
    = \frac{1}{\pi}\left(\int_{1/2}^{5/7}+\int_{5/7}^{c}\right)
    \frac{|\zeta(\sigma+iT)|^3 x^\sigma}{T}\,\mathrm{d}\sigma.
\end{equation}
Notice that for $\sigma\in[1/2,5/7]$, Lemma \ref{lemma:zeta_half} implies that the integrand on the right-hand side of \eqref{priyam} satisfies
\begin{equation*}
    \frac{\left|\zeta(\sigma+iT)\right|^3 x^\sigma }{T}  \leq \exp\left(\sigma\log{x}+3\log{\nu_{1}(\sigma,T_0)}+\left(\frac{4\lambda+3}{14}-1\right)\log{T}+3\log\log{T}\right).
\end{equation*}
Since $\lambda$ is linear in $\sigma$, it follows that $\log \nu_{1}(\sigma,T_0)$ is also linear in $\sigma$. Hence, the exponent on the right-hand side is a linear function of $\sigma$. Therefore, the right-hand side attains its maximum over the interval $[1/2,5/7]$ at one of the endpoints.
Hence, we estimate the first integral as
\begin{multline}
\label{eq:firstInt}
    \frac{1}{\pi}\int_{1/2}^{5/7}\frac{|\zeta(\sigma+iT)|^3 x^\sigma}{T}\,\mathrm{d}\sigma \\
    \leq \frac{3(\log T_1)^3}{14\pi T}
    \max\left\{
        \nu_{1}\!\left(\tfrac{1}{2},T_0\right)^3 T^{\frac{1}{2}}x^{\frac{1}{2}},\, 
        \nu_{1}\!\left(\tfrac{5}{7},T_0\right)^3 
        T^{\frac{3}{14}}x^{\frac{5}{7}}
    \right\}.
\end{multline}
Similar to the case where $\sigma\in[1/2,5/7]$, we apply Lemma~\ref{lemma:zetaMiddle2} to the second integral on the right-hand side of \eqref{priyam}. This gives 
\begin{multline}
\label{eq:secondInt}
    \frac{1}{\pi}\int_{5/7}^{c}\frac{|\zeta(\sigma+iT)|^3 x^\sigma}{T}\,\mathrm{d}\sigma \\
    \quad\leq \frac{(c-5/7)(\log T_1)^3}{\pi T}
    \max\left\{
        \nu_{2}\!\left(c,\tfrac{5}{7},T_0\right)^3 
        T^{\frac{3}{14}}x^{\frac{5}{7}},\; 
        \nu_{2}(c,c,T_0)^3 x^c
    \right\}
\end{multline}
since the integrand attains its maximum on the interval  $\sigma\in[5/7,c]$ at one of the endpoints.
Combining \eqref{eq:firstInt} and \eqref{eq:secondInt} completes the proof.
\end{proof}

\section{An integral estimate via moment and pointwise bounds }
\label{sec:momentEstimates}
In the lemmas that follow, we make use of existing explicit third moments and pointwise estimates of $\zeta(1/2+it)$ to derive estimates for certain integrals that arise in our analysis.
\begin{lemma}
\label{lemma:thirdmoment}
    Let $x>1$, $\sigma=1/2$, and $T \geq 10^7$. We have
    \begin{equation*}
|I(x,T)|:=\left|\frac{1}{2\pi i}\int_{\sigma - iT}^{\sigma + iT} \zeta^{3}(s) \frac{x^s}{s} \mathrm{d}s\right|\le x^{1/2}\Bigl(0.304\log^{\frac{7}{2}}T+1.063\log^{\frac{5}{2}}T+1.240\cdot 10^{6}\Bigr).
\end{equation*}
\end{lemma}
\begin{proof}
    Let $s=1/2+it$, then $\text{d}s=i\,\text{d}t$. We have
    \begin{align}{\label{fbi}}
        |I(x,T)|& \le \frac{x^{1/2}}{2\pi}\int_{-T}^{T}  \frac{|\zeta^{3}(1/2+it)|}{\sqrt{1/4+t^2}}\, \text{d}t=\frac{2x^{1/2}}{\pi}\int_{ 0}^{T}  \frac{|\zeta^{3}(1/2+it)|}{\sqrt{4t^{2}+1}}\,  \text{d}t.
    \end{align}
Now, we split the integration range into three parts as follows: 
\begin{align}\label{aca}
\frac{2}{\pi}\int_{ 0}^{T} \frac{|\zeta^{3}(1/2+it)|}{\sqrt{4t^{2}+1}}\text{d}t&=\frac{2}{\pi}\left(\int_{0}^{200}+\int_{ 200}^{10^{7}}+\int_{10^{7}}^{T}  \right)\frac{|\zeta^{3}(1/2+it)|}{\sqrt{4t^{2}+1}}\, \text{d}t\notag\\
       &:= I_{1}+I_{2}+I_{3}(T). 
       \end{align} 
      Next, we estimate these integrals.
      
      To estimate $I_{1}$, we compute their respective integrals using \textit{Mathematica} version 12.0.0.0. That is,
    \begin{align}\label{fghh}
    I_{1}&=\frac{2}{\pi}\int_{0}^{200} \frac{\left|\zeta^{3}(1/2+it)\right|}{\sqrt{4t^{2}+1}} \, \text{d}t\le 10.457.
        \end{align}
 For $I_{2}$, we use the estimate derived in \cite[Subsection 3.2]{HPY2022} directly. In particular,
 \begin{equation*}
     |\zeta(1/2+it)|\le 0.592t^{1/6}\log t,\ 200\le t<5.5\cdot 10^{7}.
 \end{equation*}
 This gives  
\begin{equation}\label{ewe}  
I_{2}=\frac{2}{\pi}\int_{200}^{10^{7}} \frac{\left|\zeta^{3}(1/2+it)\right|}{\sqrt{4t^{2}+1}}\,\text{d}t
\le\frac{(0.592)^{3}}{\pi}\int_{200}^{10^{7}} \frac{(\log t)^{3}}{\sqrt{t}}\, \text{d}t\le 1.240\cdot 10^{6}.
        \end{equation}
Notice that the sum of $I_{1}$ and $I_{2}$ is at most $1.240\cdot 10^{6}$.
        
 Finally, we estimate $I_{3}$ as follows. We have
\begin{align}\label{vbd}
I_{3}(T)&=\frac{2}{\pi}\int_{10^{7}}^{T} \frac{|\zeta^{3}(1/2+it)|}{\sqrt{4t^{2}+1}}\, \text{d}t.
\end{align}
To estimate the integral on the right-hand side, we  define \begin{equation*}
\phi_{m}(T):=\int_{0}^{T}|\zeta(1/2+it)|^{2m}\, \text{d}t.
\end{equation*}
Applying integration by parts to the integral on the right-hand side of \eqref{vbd} yields 
\begin{equation}\label{vbn}
I_{3}(T)\le \frac{1}{\pi}\left(\frac{\phi_{\frac{3}{2}}(T)}{T}-\frac{\phi_{\frac{3}{2}}(10^{7})}{10^{7}}+\int_{ 10^{7}}^{T}\frac{\phi_{\frac{3}{2}}(t)}{t^2}\, \text{d}t\right).
\end{equation}
By \cite[Corollary 3]{chourasiya2025explicitforminghamszero}, we write 
\begin{align}\label{ggy}
    \phi_{\frac{3}{2}}(t)&\le\left(\frac{1}{2\pi^{2}}+\frac{39.720}{\log T}+\frac{5.817\cdot 10^{12}}{T\log^{4}T}\right)^{\frac{1}{2}}T(\log T)^{5/2}\notag\\
    &\le 3.337T(\log T)^{5/2}
\end{align}
for all $T\ge T_{0}\geq 10^{7}$.
 To evaluate the right-hand side of \eqref{vbn}, we make use of the bound in \eqref{ggy}. This gives
\begin{align}\label{erb}
I_{3}(T) <0.304\log^{\frac{7}{2}}(T)+1.063\log^{\frac{5}{2}}(T)-5101.951.
\end{align}
Finally, substituting 
\eqref{fghh}, \eqref{ewe}, and \eqref{erb} into \eqref{aca} and the result into \eqref{fbi} completes the proof.
\end{proof}
\begin{remark}
    It is worth noting that the bound in \eqref{ggy} exceeds the conjectural bound for the third moment (see \cite{MR623671}) by a factor of $(\log T)^{1/4}$.
\end{remark}
\section{Effective results}
\label{sec:Effective}
Before stating the lemma, we explain its role in our argument. In order to obtain an estimate for $|\Delta_3(x)|$ that is valid only when $x$ is a half-integer or an integer, as restricting to these special values avoids the blow-up that would occur if  $x$ were too close to an integer  in Lemma \ref{lemma:dkGeneral}.
This section is devoted to explicit results from which numerical bounds may be computed for given values of $\varepsilon$, and a prescribed lower bound for $x$. The form of these results depends on both the technique employed to estimate
$\int_{\sigma-iT}^{\sigma+iT} \zeta^{3}(s)\frac{x^s}{s}\,\mathrm{d}s$,
where $\sigma=\frac12$ (see Section~\ref{sec:momentEstimates}), and on whether $x$ is an integer or a half-integer.

\subsection{Cases where $x$ is not an integer or a half-integer} 
This section follows from \cite[Section 4.1]{palojarvi2026explicit}.
We define
\begin{align*}
    u_{8}(\tau,\upsilon, j,x):=& \left(x+\frac{1}{2}\right)^{-\tau}\left(\log{\left(x+\frac{1}{2}\right)}\right)^{-\upsilon}\cdot \\
    &\cdot\begin{cases}
    \sum_{m=1}^{j}\binom{j}{m}\frac{x+1/2}{(2x)^m(\log{(x+1/2)})^{m-1}}, &\text{if } j \geq 1 \\
    \frac{1}{2}, &\text{if } j=0
    \end{cases}
\end{align*}
and
\begin{equation*}
    u_{9}(\tau, \upsilon, j, x):=
    \begin{cases}
        0, &\text{if}\, j=0 \\
        \frac{(\log{(x+1/2)})^j}{2(x+1/2)^{\tau}(\log{(x+1/2)})^{\upsilon}}, &\text{if } 1\leq j<\tau\log{\left(x+\frac{1}{2}\right)}+\upsilon \\
         \frac{1}{2}\left(\frac{j-\upsilon}{\tau}\right)^{j-\upsilon}e^{\upsilon-j}, &\text{if } j\geq \tau\log{\left(x+\frac{1}{2}\right)}+\upsilon.
    \end{cases}
\end{equation*}
 The following lemma bridges the gap by showing that the same form of estimate, with an explicitly computable constant involving the auxiliary functions $u_{8}$ and $u_{9}$, remains valid for all real $x\ge x_{1}+1/2$.
\begin{lemma}(\cite[Lemma 4.1]{palojarvi2026explicit}, $k=3$)
\label{lemma:estimateBetweenPoints}
    Assume that 
    $P_3(x)$ is as in \eqref{eq:asymptotic}, and that
    \begin{equation*}
        \left|\Delta_3(x)\right| \leq
        \begin{cases}
            \rho_{3,1}x^{\tau_{3}}(\log{x})^{\upsilon_{3}}, &\text{if } x \text{ is an integer}, \\
            \rho_{3,2}x^{\tau_{3}}(\log{x})^{\upsilon_{3}}, &\text{if } x \text{ is a half-integer} 
        \end{cases}
    \end{equation*}
    for all $ x\geq x_1>1$ that are integers or half-integers. We also suppose that $\rho_{3,1}, \rho_{3,2},\tau_{3},$ are positive real numbers and $\upsilon_3 \geq 0$. 
    
    Then for all real numbers $x \geq x_1+1/2$, we have
    \begin{equation*}
T_3(x)=xP_3(\log{x})+\Delta_{3,\text{new}}(x),
    \end{equation*}
    where
    \begin{align*}
        \left|\Delta_{3, \text{new}}(x)\right| \leq &\left(\max\{\rho_{3,1}, \rho_{3,2}\}+\sum_{j=0}^{2}|a_j|\left(u_{8}(\tau_3,\upsilon_3,j,x_1)+u_{9}(\tau_3, \upsilon_3, j, x_1)\right)\right) \\
        &\cdot x^{\tau_{3}}(\log{x})^{\upsilon_{3}}
    \end{align*}
    with $a_{0}=\beta$, $a_{1}=\alpha$, and $a_2=1/2$.
\end{lemma}
\subsection{Effective result based on the third moment}
\label{sec:EffectiveFourth}

In order to present our results, we will define $S(T)$, $u_{10}$, $u_{11}$, $u_{12}$ and $u_{13}$. Let
\begin{align}
    &S(T):=0.304\log^{\frac{7}{2}}T+1.063\log^{\frac{5}{2}}T+1.240\cdot 10^{6}, \label{def:SkNew} \\
    &u_{10}(x,T, \varepsilon, \varepsilon_1,a):=\frac{(1+\pi)x^{1+\varepsilon_1}}{\pi T}\left(u_{2}(x,\varepsilon, 1+\varepsilon_{1}, a)
    +4x(x-1/2)^{\epsilon-1-\varepsilon_{1}}+\right.\notag \\
    &\quad\left.(a+1)x\left(x+\frac{1}{2}\right)^{\varepsilon-1-\varepsilon_{1}}u_{4}(x,a)+\left(2x-\frac{1}{2}\right)\left(\frac{x}{a}\right)^{\varepsilon-1-\varepsilon_{1}}u_{3}(x,a)\right), \label{def:u9}\\
     &u_{11}(f,x, T, T_0,  T_1, \varepsilon, \varepsilon_1, a):=\frac{u_{7}\!\left(x,\tfrac{1}{2},1+\varepsilon_{1},T,T_0,T_1\right)+x^{\frac{1}{2}}S( T_1)}{x^\frac{1}{2}(\log{x})^{\frac{7}{2}}}+\notag\\
    &\quad \hspace{4cm}\frac{u_{10}(x, T, \varepsilon, \varepsilon_1,a)}{x^\frac{1}{2}(\log{x})^{\frac{7}{2}}},\label{def:u12}\\
    &u_{12}(x, T, \varepsilon, \varepsilon_1, a):=\frac{(1+\pi)x^{1+\varepsilon_1}}{\pi T}\left(\left(x-\frac{1}{2}\right)\left(\frac{x}{a}\right)^{\varepsilon-1-\varepsilon_{1}}u_{5}(x,a)
    +\right.\notag \\
    &\quad\left.\hspace{2cm}(x+1)^{\varepsilon-1-\varepsilon_{1}}\left(x+\frac{1}{2}\right)u_{6}(x,a)+u_{2}(x,\varepsilon, 1+\varepsilon_{1}, a)\right), \label{def:u13} \\
    \text{and}\notag\\
    &u_{13}(f,x, T, T_0,  T_1, \varepsilon, \varepsilon_1, a):=\frac{u_{7}\!\left(x,\tfrac{1}{2},1+\varepsilon_{1},T,T_0,T_1\right)+x^{\frac{1}{2}}S( T_1)}{x^\frac{1}{2}(\log{x})^{\frac{7}{2}}}+\notag\\
    &\quad\hspace{4cm} \frac{u_{12}(x, T, T_{0}, \varepsilon, \varepsilon_1,a)}{x^\frac{1}{2}(\log{x})^{\frac{7}{2}}},\label{def:u14}
\end{align}
where the terms from $u_1$ to $u_6$ are defined as in \eqref{u3}--\eqref{m5rr}.

Here, we provide an estimate for $\Delta_{3}(x)$ for all $x$ is small enough.
\begin{theorem}
\label{thm:EffectiveFourth}
  Assume $a \in [1.6, e]$. Let  $\varepsilon$ and $x_0$ be as in Lemma \ref{lemma:dkGeneral}, $\nu_{1}$ and $\nu_{2}$ as in Lemma \ref{Lemu7}, $0.05497<\varepsilon \le 0.4$ and $\varepsilon_1 \in(\varepsilon, 17.590)$. 

   If $x$ is a half-integer or an integer and $x\geq x_1$, then we have 
   \begin{align*}
        &\left|\Delta_3(x)\right|< \omega_1(x_1,\varepsilon, \varepsilon_1,a)x^{\frac{1}{2}}(\log{x})^{\frac{7}{2}}, 
    \end{align*}
    where 
    \begin{align*}
        \omega_1(x_1,\varepsilon, \varepsilon_1,a)&:=\max\left\{u_{11}(S,x_1, T, T_0,  T_1, \varepsilon, \varepsilon_1, a),\,  u_{13}(S,x_1, T, T_0,  T_1, \varepsilon, \varepsilon_1,a)  \right\},
    \end{align*}  
\begin{align*}
&\kappa_1=\frac{\sqrt{2}(2+7\varepsilon_1)\nu_2(1+\varepsilon_1,1+\varepsilon_1,10^{7})^3}{7\cdot 0.304(1+2\varepsilon_1)^{\frac{1}{2}}}, \text{ } 
     \kappa_2=-\frac{1}{2}, \text{ }\kappa_3=\frac{1}{2}+\varepsilon_{1}, \text{ }\\&T=\kappa_1(\log{x_1})^{\kappa_2}x_1^{\kappa_3}, \text{ } T_0=10^{7}, \text{ } T_1=\kappa_{1}x_1^{\kappa_3}
\end{align*}
    and $x_1\geq \max\{ax_0,\, e^{e}+1/2,\, 8.403\cdot 10^{15}\}$.
\end{theorem}

\begin{proof}
   Similar to the proof in \cite[Theorem 4.2]{palojarvi2026explicit}, our initial analysis is restricted to the case where $x$ is a half-integer. Furthermore, we take $\log{T} =\log{T_1}+O(\log\log{x})$ and $T \geq 10^{7}$.  Combining the estimates established in Lemmas \ref{lemma:TkIntegral}, \ref{lemma:dkGeneral}, \ref{lemma:IntegralsToBeEstimated}, and \ref{Lemu7}, we find that the dominant contributions are of the form $O(x^{c}/T)$:
    \begin{align}
        &\frac{1+\pi}{\pi}\cdot\frac{x^c}{T}\left(\left(2x-\frac{1}{2}\right)\left(\frac{x}{a}\right)^{\varepsilon-c}\log{\left(x-\frac{x}{a}-\frac{1}{2}\right)}+(a+1)x\left(x+\frac{1}{2}\right)^{\varepsilon-c}\log{x}\right. \nonumber \\
       &\hspace{2.5cm}\quad\left.+(4x-1)(x-1/2)^{\varepsilon-c}+u_{2}(x,\varepsilon,c,a)\right),\label{eq:fourthNonIntegerFirst} 
    \end{align}
     $O(T^{-\frac{1}{2}+\varepsilon'}x^{\frac{1}{2}})$, $O(T^{-\frac{11}{14}+\varepsilon'}x^{\frac{5}{7}})$ and $O(T^{-1+\varepsilon'}x^{c})$:
    \begin{multline}
    \frac{(\log T_{1})^3}{T\pi}
    \!\left[
        \frac{3}{14}
        \max\left\{
            \nu_{1}\!\left(\tfrac{1}{2},T_0\right)^3 T^{\frac{1}{2}}x^{\frac{1}{2}},\; 
            \nu_{1}\!\left(\tfrac{5}{7},T_0\right)^3 
            T^{\frac{3}{14}}x^{\frac{5}{7}}
        \right\}\right.\\
        \left.+\left(c-\frac{5}{7}\right)
        \max\left\{
            \nu_{2}\!\left(c,\tfrac{5}{7},T_0\right)^3 
            T^{\frac{3}{14}}x^{\frac{5}{7}},\; 
            \nu_{2}(c,c,T_0)^3 x^c
        \right\}
    \right],\label{cder}
\end{multline}
    and $O(x^{1/2}T^{\varepsilon'})$:
    \begin{equation}
    \label{eq:effectiveFourthHalfInteger}
        0.304x^{\frac{1}{2}}\log^{\frac{7}{2}}T.
    \end{equation}
   To optimize these resulting bound, we take $T=\kappa_1(\log{x})^{\kappa_2}x^{\kappa_3}$, and determine the constants $\kappa_1,\kappa_2, \kappa_3$ for which these bounds are optimal.

    From the terms in the middle of \eqref{cder}, we have 
    \begin{equation*}
        \frac{5}{7}-\frac{11\kappa_{3}}{14}<\min\Bigl\{\frac{1}{2}, \, 1+\varepsilon_{1}-\kappa_{3}\Bigr\}
    \end{equation*}
    for all $\kappa_{3}>3/11$.
    Hence, since $c>1+\varepsilon$, the leading terms from these bounds come from last term in \eqref{cder} and the term appearing in \eqref{eq:effectiveFourthHalfInteger}. Hence, setting $c=1+\varepsilon_{1}$ and $T_{0}=10^{7}$, we have $1+\varepsilon_1-\kappa_3=1/2$ and $3-\kappa_2=7/2$.
    Also, we have
    \begin{align}
    \label{def:kappa1FourthFirst}
       0.304\kappa_{3}^{\frac{7}{2}}&=\frac{\frac{2}{7}+\varepsilon_1}{\kappa_1\pi}(\nu_2(c,c,10^{7})\kappa_3)^3 \notag\\ \quad \kappa_1&=\frac{\sqrt{2}(2+7\varepsilon_1)\nu_2(c,c,10^{7})^3}{7\cdot 0.304\pi(1+2\varepsilon_1)^{\frac{1}{2}}}
    \end{align}
    since $\kappa_{3}=1/2+\varepsilon_{1}$.
    Hence, in our estimate for $|\Delta_3(x)|$, the power of $x$ is $1+\varepsilon_1-\kappa_3=1/2$, and the power of logarithm is $3-\kappa_2 = 7/2$.
  Since $\kappa_2<0$, we discard the logarithmic factor and choose $T_1=\kappa_1x^{\kappa_3}$ to ensure that terms $\log{T_1}/\log{x}$ are decreasing for all $x$. We now identify the leading coefficient $\omega_1$ in the estimate for $|\Delta_k(x)|$.
    
     Next, we focus on deriving the leading constant in the estimate of $|\Delta_{3}(x)|$ and show that it is bounded above by $u_{12}$ from \eqref{def:u12} if $x$ is a half-integer. In addition, we verify that $u_{11}$ is decreasing in $x$, so that its maximum value occurs at $x=x_{1}$. To begin, we consider the terms furnished by Lemma \ref{lemma:dkGeneral}. Note that
    \begin{equation}
    \label{eq:r1inr10}
        \frac{x^c u_{2}(x,\varepsilon, c,a)}{T\cdot x^{c-\kappa_3}(\log{x})^{3-\kappa_2}}=\frac{u_2(x,\varepsilon, c,a)}{\kappa_1(\log{x})^3}
    \end{equation}
    and
    \begin{align}
        &\frac{x^c\cdot (4x-1)(x-1/2)^{\varepsilon-c}}{T\cdot x^{c-\kappa_3}(\log{x})^{3-\kappa_2}}  < \frac{4x}{\kappa_1(\log{x})^3}\cdot \left(x-\frac{1}{2}\right)^{\varepsilon-c} \label{eq:maxn0}
    \end{align}
     since $x$ is a half-integer. Since the right-hand sides of \eqref{eq:r1inr10} and \eqref{eq:maxn0} decrease with $x$ for all $x>1$, their maximal values are attained at $x=x_1$. Therefore, we may replace $x$ by $x_1$ in both expressions. 
     
     Furthermore, we note that
    \begin{equation*}
        \frac{x^{c}\left(2x-\frac{1}{2}\right)u_3(x,a)}{T x^{c-\kappa_3} (\log{x})^{3-\kappa_2}}\left(\frac{x}{a}\right)^{\varepsilon-c}<\frac{2x^{1+\varepsilon}u_3(x,a)}{\kappa_1 x^c(\log{x})^3}a^{c-\varepsilon},
    \end{equation*}
    and the right-hand side is decreasing with respect to $x \geq e^e+1/2$ since $c>1+\varepsilon$, and $a \geq 1.6$. A similar idea applies to the terms coming from $u_4$ in \eqref{m3rr}. Thus, we can conclude that the right-hand side of \eqref{halfintbound} multiplied by $(1+\pi)x^c/(\pi Tx^{c-\kappa_3} (\log{x})^{3-\kappa_2})$ is at most
    \begin{equation}
    \label{eq:r10}
        \leq \frac{u_{10}(x_1, \kappa_1(\log{x_1})^{\kappa_2}x_1^{\kappa_3},\varepsilon, \varepsilon_1,a)}{ x_1^{c-\kappa_3} (\log{x_1})^{3-\kappa_2}}
    \end{equation}
    for $x \geq x_1$. This concludes the case of Lemma \ref{lemma:dkGeneral}.

    Also, we consider the terms coming from Lemma \ref{Lemu7}. Choosing $0.05497<\varepsilon \leq 0.4$, it follows immediately from \eqref{def:kappa1FourthFirst} that  $\kappa_{1}<1$.  
   Hence, for all $x \geq x_1$, we have
    \begin{multline}
    \label{eq:r11}
        \frac{u_{7}(x,1/2,c, \kappa_1(\log{x})^{\kappa_2}x^{\kappa_3},10^{7},\kappa_1(\log{x})^{\kappa_2}x^{\kappa_3})}{x^{c-\kappa_3}(\log{x})^{3-\kappa_2}} \\
        \leq \frac{u_{7}(x_1, 1/2, c, \kappa_1(\log{x_1})^{\kappa_2}x_1^{\kappa_3}, 10^{7},  x_1^{\kappa_3})}{x_1^{c-\kappa_3}(\log{x_1})^{3-\kappa_2}},
    \end{multline}
    since $\kappa_{2}<0$. 

    We consider the terms coming from Lemma \ref{lemma:thirdmoment} similarly. First, we note that $T \geq 10^{7}$, and $\kappa_{2}<0$. Now,  for $\kappa_{1}<1$, we have
    \begin{equation}
    \label{eq:SkNew}
        \frac{x^{1/2}S(T)}{x^{c-\kappa_3}(\log{x})^{3-\kappa_2}} \leq \frac{S(x_1^{\kappa_3})}{x_1^{c-\kappa_3}(\log{x_1})^{3-\kappa_2}}
    \end{equation}
 for all $x\geq x_1$, where $S(T)$ is as in \eqref{def:SkNew}. Combining estimates \eqref{eq:r10}, \eqref{eq:r11} and \eqref{eq:SkNew}, we obtain that the constant term is $u_{11}$ if $x$ is a half-integer.
 
    Finally, we consider the integer case. Note that the parameters $c$ and $T_0$ are chosen exactly as before. The only differences from the preceding analysis are that \eqref{eq:integerdSumFirst} replaces \eqref{eq:TkSum} in Lemma \ref{lemma:TkIntegral}, and \eqref{halfintbound} replaces \eqref{out} in Lemma \ref{lemma:dkGeneral}. Consequently, by Definition \eqref{fge} and the condition $T\geq T_0$, we obtain
    \begin{equation}
    \label{eq:dkIntegerCase}
        \frac{d_3(x)}{2\pi}\left(\pi+\frac{2Tc}{T^2-c^2}\right) \leq \frac{x^{\frac{1.5914\log 3}{\log\log{x}}}}{2\pi}\left(\pi+\frac{2T_0c}{T_0^2-c^2}\right). 
    \end{equation}
    The same choices of $\kappa_1$, $\kappa_2$, and $\kappa_3$ apply. For
$x \geq e^{e}+1/2$, the right-hand side of \eqref{eq:dkIntegerCase}
divided by $x^{\frac12}(\log x)^{\frac72}$ is decreasing for all
$x \geq x_1$. Hence, evaluating at $x=x_1$ yields an upper bound. The
constant term is handled as before, with $u_{10}$ replaced by $u_{12}$,
thereby giving the constant $u_{13}$.

     Lastly, since $T_0 \geq 10^{7}$ in Lemma \ref{lemma:thirdmoment}, we must have
    \begin{equation*}
       x_1 \geq \left(\frac{10^{7}(7\cdot0.304\pi(1+2\varepsilon_{1})^{1/2})}{\sqrt{2}(2+7\varepsilon_1)\nu_{2}(c,c,10^{7})^3}\right)^{2}, 
    \end{equation*}
    and by Lemma \ref{lemma:dkGeneral}, we require that $x_1 \geq e^e+1/2$. However, since $\varepsilon_1 \leq 0.44$ we have
    \begin{equation*}
        \left(\frac{10^{7}(7\cdot0.304\pi(1+2\varepsilon_{1})^{1/2})}{\sqrt{2}(2+7\varepsilon_1)\nu_{2}(c,c,10^{7})^3}\right)^{2} <\left(10^{7}(7\cdot0.304\pi(1+2\varepsilon_{1})^{1/2})\right)^{2}<8.403\cdot 10^{15}.
    \end{equation*}
     Hence, we can conclude
    \begin{equation*}
        \max\left\{8.403\cdot 10^{15},  \left(\frac{10^{7}(7\cdot0.304\pi(1+2\varepsilon_{1})^{1/2})}{\sqrt{2}(2+7\varepsilon_1)\nu_{2}(c,c,10^{7})^3}\right)^{2} \right\}=8.403\cdot 10^{15},
    \end{equation*}
    and use that in the lower bound for $x_1$.
\end{proof}
\begin{remark}
 Different explicit bounds for the third-moment of $\zeta(1/2+it)$   exist for various lower bounds on $T$. However, applying the version valid for $T\ge 1.1\cdot 10^{42}$, as in Lemma \ref{lemma:thirdmoment} and Theorem \ref{thm:EffectiveFourth}, does give an explicit bound for $\Delta_{3}(x)$,  but at the cost of a much larger constant, since the last term in the bound of Lemma \ref{lemma:thirdmoment} is itself considerably larger than the corresponding term in our case.
\end{remark}

\section{Proof of Theorem \ref{Final}}\label{Proof}
In this section, we present the proof of Theorem~\ref{Final} for  sufficiently small values of $x$, with all numerical computations carried out using \textit{Mathematica} $12.0.0.0$. Notice that we take $x_{0}=\max\{4,\ceiling{\exp(\exp(1.5914/\varepsilon))}\}$.

 From Theorem~\ref{thm:EffectiveFourth}, we set $a=1.6$ and choose $\varepsilon_{1}$ such that $x_{1}$ is small enough. In order to present the best estimate for $|\Delta_{3}(x)|$ in different ranges of $x$, we vary the choices of $\varepsilon$ and $\varepsilon_{1}$ and present the result in the table below.  
 \begin{center}
\begin{threeparttable}
\caption{Results for Theorem \ref{thm:EffectiveFourth} if $x \geq x_1$ is an integer or a half-integer.}\label{Table}
\begin{tabular}{ |c|c| } 
 \hline
 $(x_1,\varepsilon,\varepsilon_1,a)$ & $|\Delta_3(x)| <$... \\ \hline
 $(3.682\cdot 10^{31},0.371855,0.38,1.6)$ & $0.6901x^{\frac{1}{2}}(\log{x})^{\frac{7}{2}}$ \\ \hline
 $(4.133\cdot 10^{87}, 0.300000,0.31,1.6)$  & $0.2067x^{\frac{1}{2}}(\log{x})^{\frac{7}{2}}$ \\ \hline
 $(1.600\cdot 10^{98},0.293671,0.30,1
 .6)$  & $0.1947x^{\frac{1}{2}}(\log{x})^{\frac{7}{2}}$ \\ \hline
 \end{tabular}
\end{threeparttable}
\end{center} 

\section{Conclusion}\label{Conc}
In this paper, we employed the approach of \cite[Theorem 12.3]{MR882550} to derive an explicit upper bound for the error term $\Delta_3(x)$ by applying the explicit third moment estimate in \cite[Corollary 3]{chourasiya2025explicitforminghamszero}, which itself was obtained by combining the known explicit second and fourth moment estimates via H{\"o}lder's inequality. For instance, one can loosely combine an explicit pointwise bound $|\zeta(1/2+it)|\le f(t)$ with an explicit estimate for the fourth moment to obtain an estimate for the fifth moment. That is,
\begin{equation*}
    \phi_{5/2}(T)=\int_{0}^{T}|\zeta(1/2+it)|^{5}\,\mathrm{d}t\le \sup_{\substack{t\in[0,T]}}f(t)\,\phi_{2}(T).
\end{equation*}
This chain of dependencies via Perron's formula illustrates a broader principle that explicit moment estimates for $|\zeta(\frac{1}{2}+it)|$ can be combined to yield explicit bounds for higher moments, and hence explicit error term estimates for specific value of $k$ in the generalised divisor problem.

\section*{Acknowledgment}
I would like to thank Andrew Fiori and Neea Paloj\"arvi for suggesting this research direction, and Neea Paloj\"arvi for her continued support, insightful discussions, and constructive feedback throughout this work. I also thank Tim Trudgian for valuable discussions that contributed to the improvement of this manuscript.
\printbibliography
\end{document}